\documentclass[12pt,a4paper,reqno]{amsart}
\usepackage{amssymb}
\usepackage{amscd}
\usepackage{enumerate}
\usepackage{graphicx}
\usepackage{siunitx}
\usepackage{tikz-cd}
\usepackage{color}
\usepackage[all]{xy}           %xypic macro for latex2.09
\usepackage{graphicx}
\usepackage{color}

\numberwithin{equation}{section}

\newtheorem{thm}{Theorem}[section]
\newtheorem{cor}[thm]{Corollary}
\newtheorem{lem}[thm]{Lemma}
\newtheorem{prop}[thm]{Proposition}
\newtheorem{defn}[thm]{Definition}

\theoremstyle{definition}
\newtheorem{rem}[thm]{Remark}

\newcommand{\ep}{\epsilon}

\newcommand{\cc}{\mathbb{C}}

\newcommand{\z}{\mathbb{Z}}

\newcommand{\Ind}{\mathrm{Ind}}

\newcommand{\Vir}{\mathrm{Vir}}

\newcommand{\SVir}{\mathrm{SVir}}

\newcommand{\p}{\mathcal{P}}
\newcommand{\q}{\mathcal{Q}}

\begin{document}

\title[]{Whittaker modules  for the super-Virasoro algebras}

\author[Liu]{Dong Liu}
\address{Department of Mathematics, Huzhou University, Zhejiang Huzhou, 313000, China}
\email{liudong@zjhu.edu.cn}

\author[Pei]{Yufeng Pei}
\address{Department of Mathematics, Shanghai Normal University,
Shanghai, 200234, China} \email{pei@shnu.edu.cn}

\author[Xia]{Limeng Xia}

\address{Institute of Applied System Analysis, Jiangsu University, Jiangsu Zhenjiang, 212013,
China}\email{xialimeng@ujs.edu.cn}


\maketitle

\begin{abstract}
In this paper, we define and study Whittaker modules for the super-Viraoro algebras, including the Neveu-Schwarz algebra and the Ramond algebra. We classify the simple Whittaker modules and obtain necessary and sufficient conditions for  irreducibility of these  modules.

%\noindent{\bf Keywords:} Super-Virasoro algebras, Whittaker modules.

\end{abstract}

\section{Introduction}

Whittaker vectors and Whittaker modules play a critical role in the representation theory of finite-dimensional simple Lie algebras (cf.~\cite{AP,K}). Recently  Whittaker modules have been intensively studied for many infinite dimensional Lie algebras such as the Virasoro algebra \cite{LGZ,FJK,OW1,OW2},  Heisenberg algebras \cite{Ch}, and  affine Kac-Moody algebras \cite{ALZ}. Analogous results in similar settings have been worked out for many Lie algebras with triangular decompositions (cf.~\cite{CSZ, LWZ,W,TWX,ZTL} and references therein). A general categorial framework for Whittaker modules was proposed in \cite{BM,MZ}. Degenerate Whittaker vectors of the Virasoro algebra naturally appear in  the AGT conjecture  in physics \cite{G}. An explicit formula for degenerate Whittaker vectors  has been obtained in terms of the Jack symmetric functions in \cite{Y}.

Whittaker modules and Whittaker categories  have been generalized to  finite-dimensional
simple Lie superalgebras  based on some representations of nilpotent finite-dimensional Lie superalgebras in \cite{BCW}.  In this paper, we  define and study  Whittaker  modules for the super-Viraoro algebras, which are  Lie superalgebras and known in literature under the name the N=1 superconformal algebras \cite{BMRW,NS,Ra}. It is known that they can be viewed as certain supersymmetric extensions of the Virasoro algebra and arise as the covariant constraints in the classical formulation of the RNS model. Representations for  the super-Viraoro algebras  have been extensively investigated (cf.~\cite{IK1,IK2,KV,KW,Su}).
We  classify all finite-dimensional simple modules over certain subalgebra over the super-Virasoro algebra. Furthermore  we classify the simple Whittaker modules and obtain necessary and sufficient conditions for  irreducibility of these modules.
It is worth remarking that  degenerate Whittaker vectors of the super-Virasoro algebras  have been investigated in \cite{DLM}.

The aforementioned results demonstrate that the Whittaker modules defined in present paper
satisfy some properties that their non-super analogues do. However, there are several
differences and some features that are new in the super case. It has been observed in \cite{BCW}, simple finite-dimensional modules for a finite-dimensional nilpotent Lie superalgebra are not always one-dimensional \cite{Se}.  This leads to
an additional challenge for generalizing Lie algebra results in the Lie superalgebra setting. In our settings,
finite-dimensional simple modules over the positive parts of the super-Virasoro algebras are proved to be two-dimensional, in contrast to the situation for the Virasoro algebra. For this reason, our definition of Whittaker modules is closely related to certain smaller subalgebras.

The paper is arranged as follows. In Section 2, we recall some
notations and collect known facts about the super-Viraoro algebras. In Section 3, we classify all
finite-dimensional simple modules over certain subalgebra of the super-Viraoro algebras.
In Section 4,  we classify all simple Whittaker modules and obtain necessary and sufficient conditions for  irreducibility of these modules.

Throughout this paper, we shall use $\cc,\cc^* , {\mathbb N}$, $\z_+$ and $\z$ to denote the sets of complex numbers, nonzero complex numbers, non-negative integers, positive integers and integers respectively.
For convenience,  all elements in superalgebras and modules are homogenous unless specified.

\section{Preliminaries} \label{sec:preliminaries}
In this section, we introduce the notation and conventions that will
be used throughout the paper.

Let $V = V_{\bar 0}\oplus V_{\bar 1}$ be any $\z_2$-graded vector space. Then any element $u\in V_{\bar 0}$ (${\rm resp. }$
$u\in V_{\bar 1}$) is said to be even (${\rm resp.}$ odd). We define $|u|=0$ if $u$ is even and $|u|=1$ if $u$ is odd. Elements in $V_{\bar 0}$ or $V_{\bar 1}$ are called homogeneous. For convenience all elements in superalgebras and modules are homogenous unless specified throughout this paper.

Let $\mathfrak{g}$ be a  Lie superalgebra, a $\mathfrak{g}$-module is a $\z_2$-graded vector space $V$ together with a bilinear map  $\mathfrak{g}\times V\to V$, denoted $(x,v)\mapsto xv$ such that
$$
x(yv)-(-1)^{|x||y|}y(xv)=[x,y]v
$$ and  $$\mathfrak g_{\bar i} V_{\bar j}\subseteq V_{\bar i+\bar j}$$
for all $x,y\in \mathfrak{g}, v\in V$. Thus there is a parity-change functor $\Pi$ on the category of $\mathfrak{g}$-modules, which interchanges
the $\z_2$-grading of a module. We use $U(\mathfrak{g})$ to denote the universal enveloping algebra.

All modules for Lie superalgebras considered in this paper  are $\z_2$-graded and all simple modules are nontrivial.

\begin{defn}\label{Defi-SVA}
The {\bf super-Virasoro algebras} are the Lie superalgebras
$$
\SVir_{\ep}=\bigoplus_{n\in\z}\cc L_m\oplus\bigoplus_{r\in\z+\ep} \cc G_r\oplus\cc  C
$$
which satisfies the following commutation relations:
\begin{align*}
[L_m, L_n] &= (m-n)L_{m+n} + \delta_{m,-n} \frac{m^3-m}{12} C, \\
[L_m, G_r] &= \left({m\over2}-r\right)G_{m+r},\\
[G_r, G_s] &= 2L_{r+s}+\frac{1}{3}\delta_{r+s,0}\left(r^2-\frac{1}{4}\right)C,\\
[\SVir_\ep,C] &= 0,
\end{align*}
for all $m, n\in\z,\ep=\frac{1}{2},0$, $r, s\in\z+\ep$. $\SVir_{0}$ is called the {\bf Ramond algebra} and $\SVir_{\frac{1}{2}}$ is called
the {\bf Neveu-Schwarz algebra}.
\end{defn}

By definition, we have the following decompositions:
$$
\SVir_\ep=\SVir_\ep^{\bar0}\oplus\SVir_\ep^{\bar1},
$$
where
$$
\SVir_\ep^{\bar0}=\bigoplus_{n\in\z}\cc L_m\oplus\cc  C,\quad \SVir_\ep^{\bar1}=\bigoplus_{r\in\z+\ep} \cc G_r.
$$
It is clear  that $\SVir_\ep^{\bar 0}$ is isomorphic to the well-known Virasoro algebra $\Vir$.

The Neveu-Schwarz algebra $\SVir_\ep$ has a $(1-\ep)\z$-grading by the eigenvalues of the adjoint action of $L_0$.
It follows that $\SVir_\ep$ possesses the following triangular decomposition:
$$
\SVir_\ep=\SVir_\ep^+\oplus \SVir_\ep^0\oplus \SVir_\ep^-
$$
where
$$
\SVir_\ep^\pm=\bigoplus_{n\in\mathbb{N}}\cc L_{\pm n}\oplus \bigoplus_{r\in\mathbb{N}+\ep}\cc G_{\pm r},\quad \SVir_\ep^0=\cc L_0\oplus\cc \delta_{\ep,0}G_0\oplus\cc C.
$$

 Set
\begin{eqnarray*}
\mathfrak{p}_\ep=\bigoplus_{n\geq 1}\cc L_n\oplus \bigoplus_{n\geq2}\cc G_{n-\ep}.
\end{eqnarray*}
It is clear that $\mathfrak{p}_\ep=\mathfrak{p}_\ep^{\bar0}\oplus\mathfrak{p}_\ep^{\bar1}$ is a subalgebra of $\SVir^+$, where
 $$
 \mathfrak{p}_\ep^{\bar0}=\bigoplus_{n\geq1}\cc L_n,\quad \mathfrak{p}_\ep^{\bar1}=\bigoplus_{n\geq 2}\cc G_{n-\ep}.
 $$

\begin{defn}\label{defi-w}For $c\in \cc$, let $\psi:\mathfrak{p}_\ep\to\cc$ be a Lie superalgebra homomorphism. A $\SVir_\ep$-module $M_\ep$ is called a {\bf Whittaker module} of type $(\psi,c)$ if
\begin{itemize}
\item[(i)] $M_\ep$ is generated by a homogeneous vector $w$;

\item[(ii)]$xw=\psi(x)w$ for any $x\in\mathfrak{p}_\ep$;

\item[(iii)]$Cw=cw$,

\end{itemize}
 where $w$ is called a {\bf Whittaker vector}  of $M$.
\end{defn}

Let  $\psi:\mathfrak{p}_\ep\to\cc$ be a Lie superalgebra homomorphism. Let $\cc w_\psi$ be one-dimensional $\mathfrak{p}_\ep$-module with $xw_\psi=\psi(x)w_\psi$ for any $x\in\mathfrak{p}_\ep$ and $Cw_\psi=cw_\psi$ for some $c\in\cc$. Define induced module
$$
W_\ep(\psi,c)=:\Ind_{\mathfrak{p}_\ep\oplus\cc C}^{\SVir_\ep}\cc w_{\psi}= U(\SVir_\ep) \otimes_{U(\mathfrak{p}_\ep\oplus\cc C)}\cc w_{\psi}.
$$
Then $W_\ep(\psi,c)$ is a Whittaker module of level $c$ for $\SVir_\ep$.

Clearly $W_\ep(\psi,c)$ contains a unique maximal submodule, and its irreducible quotient is denoted by $L_\ep(\psi,c)$.

\begin{rem}
When $\psi$ is trivial ($\psi=0$), the Whittaker module $W_\ep(0, c)$  and any $h\in\cc$, $L_0w-hw$ is a Whittaker vector. Moreover $W_\ep(0, c)/<L_0-hw>$ is the standard Verma module over the super Virasoro algebra $\SVir_\ep$, which were studied in \cite{IK1}, \cite{IK2}, etc..
%So throughout this paper we always suppose that $\psi$ is nontrivial if we consider the Whittaker module $W_\ep(\psi, %c)$.
\end{rem}

We define a {\it pseudopartition} $\lambda$ to be a non-decreasing
sequence of non-negative integers
\begin{equation} \label{eqn:pseudopart1}
\lambda=( 0 \leq \lambda_1 \leq \lambda_2 \leq \cdots \leq
\lambda_m).
\end{equation}
Denote by $\p$ the set of pseudopartitions. Similarly, we define a {\it strict pseudopartition} $\lambda$ to be a strict increasing
sequence of non-negative integers
\begin{equation} \label{eqn:pseudopart1}
\lambda=( 0 \leq \lambda_1 < \lambda_2 < \cdots <
\lambda_m).
\end{equation}
Denote by $\q$ the set of strict pseudopartitions. We also introduce an alternative notation for pseudopartitions. For
$\lambda \in \p$ or $\q$, write
\begin{equation}\label{eqn:pseudopart2}
\lambda = (0^{\lambda(0)}, 1^{\lambda(1)}, 2^{\lambda(2)}, \ldots),
\end{equation}
where $\lambda(k)$ is the number of times $k$ appearing in the
pseudopartition and $\lambda(k)=0$ for $k$ sufficiently large. Clearly if $\mu \in \q$, then $\mu(k)\in\{0, 1\}$.

For $\lambda\in\p,\mu\in \q$, define elements $L_{-\lambda},
G_{\pm\ep-\mu} \in U(\SVir_\ep)$ by
\begin{eqnarray*}
L_{-\lambda} &=& L_{-\lambda_s} \cdots L_{-\lambda_2}  L_{-\lambda_1}= \cdots L_{-1}^{\lambda(1)} L_0^{\lambda(0)},\\
G_{\pm\ep-\mu} &=& G_{\pm\ep-\mu_r}\cdots G_{\pm\ep-\mu_2} G_{\pm\ep-\mu_1} = \cdots G_{\pm\ep-2}^{\mu(2)} G_{\pm\ep-1}^{\mu(1)}G_{\pm\ep}^{\mu(0)}.
\end{eqnarray*}

Define
\begin{eqnarray*}
&&|\lambda |= \lambda_1 +  \lambda_2 + \cdots + \lambda_s=\sum_{i\ge0}i\lambda(i) \quad \mbox{(the size of $\lambda$)},\\
&&|\mu\pm\ep|=(\mu_1\pm\ep)+(\mu_2\pm\ep)+\cdots+(\mu_r\pm\ep).
\end{eqnarray*}

Define $\underline{0}= (0^0, 1^0, 2^0, \ldots)$, and write
$L_{\underline{0}}=1 \in U(\SVir_\ep)$. We will consider $\underline{0}$ to
be an element of $\p$. For any $\lambda\in \p$ and $\mu\in\q$, $L_{-\lambda}G_{\pm\ep-\mu}  \in
U(\SVir_\ep)_{-|\lambda|-|\mu\mp\ep|}$, where $U(\SVir_\ep)_{-|\lambda|-|\mu\mp\ep|}$ is the
$-|\lambda|-|\mu\mp\ep|$-weight space of $U(\SVir_\ep)$ under the adjoint action.
In particular, if $\lambda \in \p, \mu\in\q$, then $L_{-
\lambda}G_{\pm\ep-\mu} \in U(\SVir_\ep^-)_{-|\lambda|-|\mu\mp\ep|}$.

\section{Simple modules over $\SVir_\ep^+$}

In this section we classify all finite dimensional simple modules over the subalgebras $\SVir_\ep^+$ and $\mathfrak{p}_\ep$.

\begin{lem}\label{lie-hom-1}
Let $\psi:\mathfrak{p}_\ep\to \cc$ be a Lie superalgebra homomorphism. Then
\begin{eqnarray*}
\psi(L_n)=0,\quad  \psi(G_r)=0
\end{eqnarray*}
for $n\in\z, n\geq3 $ and $ r\in \z+\ep, r\geq 2-\ep$
\end{lem}
\begin{proof}
It follows immediately from Definition \ref{Defi-SVA}.
\end{proof}

\begin{lem}
Let $\psi:\mathfrak{p}_\ep\to \cc$ be a Lie superalgebra homomorphism and
 $A_\ep(\psi)=\cc w_{\psi}\oplus \cc u_{\psi}$ with
 be a two-dimensional vector space with
\begin{eqnarray*}
x w_\psi&=&\psi(x)w_\psi, \
G_{1-\ep} w_\psi=u_\psi,\quad \forall x\in\mathfrak{p}_\ep.
\end{eqnarray*}
Then
\begin{itemize}
\item[(1)]$A_\ep(\psi)$ is a $(1|1)$-dimensional $\SVir_\ep^+$-module.
\item[(2)]$A_{\frac{1}{2}}(\psi)$ is simple if and only if $\psi$ is non-trivial.
\item[(3)]$A_{0}(\psi)$ is simple if and only if $\psi(L_2)\neq0$.
\item[(4)]$A_{0}(\psi)$ with $\psi(L_2)=0$ and $\psi(L_1)\neq0$ has a nontrivial submodule spanned by the vector $u_{\psi}$.
\item[(5)]$A_{\ep}(0)$ has a trivial submodule spanned by the vector $u_{\psi}$.
\end{itemize}
\end{lem}
\begin{proof}It is straightforward to verify.
\end{proof}

%\begin{lem}[\cite{BM}] \label{lemma-vir} Let $\Vir^+=\bigoplus_{n\in\z_+}\cc L_n$. Then every finite-dimensional simple modules over $\Vir^+$ is one-dimensional.
%\end{lem}
%\begin{proof}
%For the reader's convenience,  we will give an elementary proof. Let $V$ be a finite-dimensional simple $\Vir^+$-module.
%Consider the Lie algebra homomorphism
% $\varphi: \Vir^+\to {\rm End}_\cc V$ induced by the action $\Vir^+$ on $V$. It is clear  that ${\rm Ker}\,\varphi$ is infinite-dimensional since ${\rm End}_\cc V$ is finite-dimensional.
%Set $x=\sum a_iL_i\in {\rm Ker}\,\varphi\, (\ne 0)$ such that $xV=0$.
%By actions of suitable $L_i's$, there exists $L_k$ such that $L_kV=0$, and then $L_nV=0$ for all $n\ge k$. In this case the quotient of $\Vir^+$ by the kernel of this action is finite-dimensional and  all $L_i$ are locally nilpotent on $V$.  It follow from Engel's Theorem  that there exists a nonzero $v\in V$ such that $L_iv=c_iv$ for some $c_i\in\cc$ and for all $i\in\z_+$.  By the irreducibility of $V$, it has to be one-dimensional.
%\end{proof}

\begin{prop} \label{irr-finite}
Assume $V$ is a finite-dimensional simple module over $\SVir_\ep^+$. Then $V$ is isomorphic to $A_\ep(\psi)$ or its simple sub(or quotient)-module, up to parity-change, where  $\psi:\mathfrak{p}_\ep\to \cc$ is a non-trivial Lie superalgebra homomorphism.
\end{prop}
\begin{proof} Suppose that $V$ is a finite-dimensional simple $\SVir_{\ep}^+$-module.  $V$ can be regarded as a finite-dimensional $\Vir^+$-module. Let $W$ be a simple $\Vir^+$-submodule of $V$. It follows from Proposition 12 in \cite{MZ} that $W$  has to be  one-dimensional.  Then  $W=\cc w_\psi$  and $L_nw_\psi=c_nw_\psi$ for all $n\in\z_+$, where $c_1, c_2\in\cc$ and $c_i=0$ for all $i\ge3$.

Next we always assume  that $\ep=\frac{1}{2}$ since the proof of $\ep=0$ case  can be managed by essentially the same way.

If  $G_{\frac12}w_\psi=0$, $V$ becomes a trivial $\SVir_{\frac{1}{2}}^+$-module. Assume that $G_{\frac12}w_\psi\ne0$. Let $U$ be a vector space spanned by
$
u_i=G_{i+\frac12}w_\psi
$
for $i\in \mathbb{N}$.

{\bf Claim 1.}\  $U$ is a finite-dimensional $\Vir^+$-module.

It follows from
$$
L_nu_i=\psi(L_n)u_{i}+(\frac{n}{2}-i-\frac{1}{2})u_{n+i}\in U,\quad n\in\z_+, i\in \mathbb{N}.
$$

{\bf Claim 2.}\  Let $U_i$ be the vector space spanned by $u_0,\cdots, u_i$. If there $m\in \mathbb{Z}_+$ such that $u_m\in U_{m-1}$, $u_{m+k}\in U_{m-1}$ for $k\in \z_+$.

In fact, if
\begin{equation}\label{eq1}
u_m=a_0u_0+a_1u_1+\cdots+a_{m-1}u_{m-1}
\end{equation}
for $a_i\in \cc$, for $k=1$, we get
$$
u_{m+1}=-\frac{1}{m}(\sum_{i=0}^{m-1} a_iL_1u_i)+\psi(L_1)u_m\in U_{m-1}.
$$
By induction on $k$, we  have $u_{m+k}\in U_{m-1}$ for $k\in\z_+$.

{\bf Claim 3.}\ There exists a positive integer $m\in\z_+$ such that $u_m=0$.

Since $U$ is finite-dimensional,  we can choose minimal $p\in\z_+$  such that
\begin{equation}\label{eqn=1}
a_{i_1}u_{i_1}+a_{i_2}u_{i_2}+\cdots+a_{i_p}u_{i_p}=0,\end{equation}
for  $a_{i_1}, a_{i_2}, \cdots, a_{i_p}\in\cc^*$, where $i_1<i_2<\cdots <i_p$. Then by action of $L_{2i_1+1}$ on both sides of (\ref{eqn=1}) we get
\begin{equation}\label{eqn=2}
(i_1-i_2)a_{i_2}u_{2i_1+1+i_2}+(i_1-i_3)a_{i_3}u_{2i_1+1+i_3}+\cdots+(i_1-i_p)a_{i_p}u_{2i_1+1+i_p}=0,
\end{equation}
 which implies $p=1$. It gets the Claim 3.

By Claim 3, there exists $n_0\in\z_+$ such that $G_{n_0+\frac12}w_\psi=0$. Furthermore we get $G_{n+\frac12}w_\psi=0$ for all $n\ge n_0$ by actions of $L_1$. Choose minimal $m\in\z_+$ such that $G_{m-\frac12}w_\psi\ne 0$ and $G_{n+\frac12}w_\psi=0$ for all $n\ge m$. In this case
if we set $w'=G_{m-\frac12}w_\psi$, then $G_{m-\frac12}w'=0$ and $G_{n+\frac12}w'=0$ for all $n\ge m$, which means $m=1$. Then
$$V=\cc w_\psi \oplus \cc u_\psi\cong A_{\frac{1}{2}}(\psi).$$
\end{proof}

\begin{prop} \label{irr-finite-1}
Any finite-dimensional simple module over $\mathfrak{p}_\ep$ is one-dimensional.
\end{prop}
\begin{proof} For $\ep=\frac{1}{2}$, suppose that $V$ is a finite-dimensional simple $\mathfrak{p}_{\frac{1}{2}}$-module.
Then there exists $w\in V$ such that $L_nw=c_nw$ for all $n\in\z_+$, where $c_1, c_2\in\cc$ and $c_i=0$ for all $i\ge3$. With the same proof in Proposition \ref{irr-finite}, we have
$$
G_rw=0,\quad r\ge \frac32.
$$
Then $V=\cc w$.  Since the proof for  $\ep=0$ case is similar, we omit the details.
\end{proof}

\begin{cor}Let $V$ be finite-dimensional simple module over $\SVir_\ep^+$ with $C$ acts as a scalar $c\in \cc$.  Define induced $\SVir_\ep$-module
$$
M_\ep(V,c)=U(\SVir_\ep)\otimes_{U(\SVir_\ep^+\oplus\cc C)}V.
$$
Then there exists  a Lie superalgebra homomorphism $\psi:\mathfrak{p}_\ep\to\cc$ such that
$$
M(V,c)\cong W_\ep(\psi,c)
$$
as  modules over $\SVir_\ep$.
\end{cor}

\section{Whittaker modules}
In this section,   we classify all Whittaker vectors of $W_\ep(\psi,c)$ and obtain a criterion of simplicity of $W_\ep(\psi,c)$.

\begin{lem}
$\{L_{- \lambda}G_{\ep-\mu}w_{\psi}\mid \lambda\in\p, \mu\in\q\}$ forms a basis of the Whittaker module $W_{\ep}(\psi,c)$.
\end{lem}
\begin{proof}
It is a consequence of the PBW theorem.
\end{proof}

\begin{lem}\label{deri} Let $w_\psi\in W_{\ep}(\psi,c)$ be the Whittaker vector.
For $v=uw_\psi\in W_{\ep}(\psi,c)$, we have
$$
(x-\psi(x))v=[x, u]w_\psi,\quad\forall x\in\mathfrak{p}_\ep.
$$
\end{lem}
\begin{proof} For $x\in\mathfrak{p}_\ep$, by Lemma \ref{lie-hom-1} we have
$$
[x,u]w_\psi=xuw_\psi\pm uxw_\psi=xuw_\psi\pm\psi(x)uw_\psi=(x-\psi(x))uw_\psi=(x-\psi(x))v.
$$
\end{proof}

For
$$
v=\sum_{(\lambda,\mu)\in \p\times\q}p_{\lambda,\mu}L_{- \lambda}G_{\ep-\mu}w_{\psi}\in W_{\ep}(\psi,c),
$$
we define
$$
{\rm maxdeg}(v)=\max\{|\lambda|+|\mu-\ep|\mid p_{\lambda,\mu}\neq0\}
$$

\begin{lem}\label{lemm-weight-1} For any
$(\lambda, \mu)\in \p\times\q,$ $k\in \mathbb{Z}_+,$ we have

\begin{itemize}
\item[(i)] {\rm maxdeg}$([G_{k+2-\ep},
L_{-\lambda}G_{\ep-\mu}]w_{\psi})\leq |\lambda|+|\mu-\ep|-k+\ep;$
\item[(ii)] If $\mu(i)=0$ for all $0\leq
i\leq k$, then
$${\rm maxdeg}([G_{k+2-\ep},
L_{-\lambda}G_{\ep-\mu}]w_{\psi}) \leq|\lambda|+|\mu-\ep|-k-1+\ep;$$
\item[(iii)] If $\mu(i)=0$ for all $i\leq k$ while $\mu(k)=1$, then
$$[G_{k+2-\ep},
L_{-\lambda}G_{\ep-\mu}]w_{\psi}
=v+2\mu(k)\psi(L_{2})
L_{-\lambda}G_{\ep-\mu'}w_{\psi},
$$
$$
{\rm maxdeg}(v)\le|\lambda+ \mu|-k-1+\ep,
$$
and
$\mu{'}$ satisfies that $\mu'(i)=\mu(i)$ for all $i\neq k$ and $\mu'(k)=0$ and then
$$
{\rm maxdeg}(L_{-\lambda}G_{\frac12-\mu'}w_{\psi})=|\lambda|+|\mu-\ep|-k+\ep.
$$
\end{itemize}

\end{lem}
\begin{proof} Assume that $\ep=\frac{1}{2}$.

To prove (i), we write $[G_{\frac12+k},
L_{-\lambda}G_{\frac12-\mu}]$ as a linear combination of the basis of $U(\SVir_{\frac{1}{2}}):$
\begin{eqnarray*}
[G_{\frac{3}{2}+k},
L_{-\lambda}G_{\frac12-\mu}]w_{\psi}=[G_{\frac32+k},
L_{-\lambda}]G_{\frac12-\mu}w_{\psi}\pm L_{-\lambda}[G_{\frac32+k},
G_{\frac12-\mu}]w_{\psi}.
\end{eqnarray*}
For the second summand, due to $[G_{\frac32+k}, G_{\frac12-k}]w_{\psi}=2\psi(L_2)w_{\psi}$, then the maxdeg of second summand
maybe become $|\lambda|+|\mu-\frac12|-k+\frac{1}{2}$. (ii) follows from the calculations as the second summand in (i). Combining with (i) and (ii), we get (iii).

Since the proof for $\ep=0$ case is similar, we omit the details.
\end{proof}

\begin{prop}\label{whittaker-vector} Suppose that $\psi$ is nontrivial and $w_\psi\in W_{\ep}(\psi,c)$ be the Whittaker vector.
\begin{itemize}
\item[(i)]If $\psi(L_2)\ne0$, then $w'\in W_{\frac{1}{2}}(\psi,c)$ is a Whittaker vector if and
only if $w'\in \mbox{span}_{\cc}\{w_\psi\}$.
\item[(ii)]If $\psi(L_2)=0$ and $\psi(L_1)\ne0$, then $w'\in W_{\frac{1}{2}}(\psi,c)$ is a Whittaker vector if and
only if $w'\in \mbox{span}_{\cc}\{w_\psi,G_{\frac12}w_\psi\}$.
\item[(iii)] $v\in  W_{0}(\psi,c)$ is a Whittaker vector if and
only if $w'\in \mbox{span}_{\cc}\{w_\psi,G_{1}w_\psi\}$.
\end{itemize}
 \end{prop}
\begin{proof} Let $w'=uw\in W_{\frac{1}{2}}(\psi,c)$ be a Whittaker vector, we can write $w'$ a linear combination
of the basis of $W_{\frac{1}{2}}(\psi,c)$:
$$w'=\sum_{(\lambda, \mu)\in\p\times\q}
p_{\lambda, \mu} L_{-\lambda}G_{\frac12-\mu}w,$$ where $p_{\lambda,
\mu}\in \cc.$ Set
$$N:={\rm max}\{|\lambda|+|\mu-\frac12|\,|\,
p_{\lambda, \mu}\neq0\},$$
$$\Lambda_N:=\{(\lambda, \mu)
\in\p\times\q\,|\, p_{\lambda, \mu}\neq0, |\lambda|+|\mu-\frac12|=N\}.$$

Let $\psi':\mathfrak{p}_{\frac{1}{2}}\to\cc$ be a Lie superalgebra homomorphism and $\psi'\ne \psi$. Then there exists
at least one of $\{L_1, L_2\},$ denoted by $E,$ such
that $\psi(E)\neq \psi'(E).$ Assume that $w'$ is a Whittaker vector
of $W_{\ep}(\psi', c)$. On the one hand,
$$Ew'=v+\sum_{(\lambda, \mu)\in\Lambda_{N}}
p_{\lambda, \mu} L_{-\lambda}G_{\frac12-\mu}\psi(E)w,$$ where
maxdeg$(v)<N.$ On the other hand, $Ew'=\psi'(E)w'$. Thus
$\psi'(E)=\psi(E),$ which is contrary to our assumption that
$\psi'(E)\neq\psi(E).$

 Next, for $w',$ we will show that if there is
$(\bar{0},\bar{0})\neq(\lambda, \mu)\in \p\times\q$ such that
$p_{\lambda, \mu}\neq 0,$ then there is $E_n\in\{L_n, G_{n+\frac12}\mid n\ge1\}$ such that
$(E_n-\psi(E_n))w'\neq0$. In this case $w'$ is not a Whittaker
vector, which proves the necessity.

 Assume that
$p_{\lambda, \mu}\neq 0$ for some $(\lambda, \mu)\neq
(\bar{0},\bar{0})$, then, by Lemma \ref{deri}, we have
$$
(E_n-\psi(E_n))w'=\sum_{\lambda, \mu} p_{\lambda, \mu} [E_n,
L_{-\lambda}G_{\frac12-\mu}]w.
$$

{\bf Case (i)}: $\psi(L_2)\ne0$.

Suppose that there exists $(\lambda, \mu)\in \Lambda_N$ with $\mu(n)=1$ for some $n\in\mathbb{N}$.
Now we can set $k:={\rm min}\{n\in \mathbb{N}|\,
\mu(n)=1\  {\rm \ for \  some }\ (\lambda,
\mu)\in \Lambda_N\}$, then
\begin{eqnarray*}
G_{k+\frac32}w'&=&\sum_{(\lambda, \mu)\notin\Lambda_N} p_{\lambda, \mu}
[G_{k+\frac32},
L_{-\lambda}G_{\frac12-\mu}]w\\
&&+\sum_{\begin{subarray}\ (\lambda, \mu)\in\Lambda_N\\
\ \ \mu(k)=0
\end{subarray}} p_{\lambda, \mu}
[G_{k+\frac32},
L_{-\lambda}G_{\frac12-\mu}]w\\
&&+\sum_{\begin{subarray}\ (\lambda, \mu)\in\Lambda_N\\
\ \ \mu(k)=1
\end{subarray}} p_{\lambda, \mu}
[G_{k+\frac32}, L_{-\lambda}G_{\frac12-\mu}]w.
\end{eqnarray*}
By using Lemma \ref{lemm-weight-1} (i) to the first summand, we know that the maxdeg
of it is strictly smaller than $N-k+\frac12$. As for the second summand,
note that $\mu(i)=0$ for $0\leq i\leq k,$ the maxdeg of $[G_{k+\frac32}, L_{-\lambda}G_{\frac12-\mu}]$
is also strictly smaller than $N-k$ by Lemma \ref{lemm-weight-1} (ii). Now using
Lemma \ref{lemm-weight-1} (iii) to the third summand, we know it has form:
$$v+2\sum_{\begin{subarray}
\ \ (\lambda, \mu)\in\Lambda_N\\
\ \ \mu(k)=1
\end{subarray}}\mu(k)\psi(L_{2})p_{\lambda, \mu}
L_{-\lambda}G_{\frac12-\mu'}w,$$ where maxdeg$(v)\leq N-k-\frac12$, and ${\mu'}$ satisfies
$\mu'(k)=0$, $\mu'(i)=\mu(i)$ for all
$i>k$. So the maxdeg of the third summand is $N-k+\frac12$. This implies that
$G_{k+\frac32}w'\neq0$.

Now we can suppose that $\mu(i)=0$ for all $i\ge0$ if $(\lambda,\mu)\in \Lambda_N$. Then we can set
$l:={\rm min}\{n\in \mathbb{N}\,|\,
\lambda(n)\neq0\  {\rm \ for \  some }\ (\lambda,
\mu)\in \Lambda_N\}$.  Then
\begin{eqnarray*}
(L_{l+2}-\psi(L_{l+2}))w'&=&\sum_{(\lambda, \mu)\notin\Lambda_N} p_{\lambda, \mu}
[L_{l+2},
L_{-\lambda}G_{\frac12-\mu}]w\\
&&+\sum_{\begin{subarray}
\ (\lambda, \mu)\in\Lambda_N\\
\ \ \lambda(l)=0
\end{subarray}} p_{\lambda, \mu}
[L_{l+2},
L_{-\lambda}]w\\
&&+\sum_{\begin{subarray}
\ (\lambda, \mu)\in\Lambda_N\\
\ \ \lambda(l)\neq0
\end{subarray}} p_{\lambda, \mu}
[L_{l+2}, L_{-\lambda}]w.
\end{eqnarray*}
Since the maxdeg's of the first summand and second summand are strictly smaller than $N-l$, and the maxdeg's of the third summand is $N-l$, we have $(L_{l+2}-\psi(L_{l+2}))w'\ne 0$ and get the proposition.

\noindent{\bf Case (ii).} $\psi(L_2)=0$ and $\psi(L_1)\ne0$.

In this case, we can also get $\mu(i)=0$ for all $i\ge1$ if $(\lambda,\mu)\in \Lambda_N$ by action of $G_{k+\frac12}$ as in Case I.

Set
$l:={\rm min}\{n\in \mathbb{N}\,|\,
\lambda(n)\neq0\  {\rm \ for \  some }\ (\lambda,
\mu)\in \Lambda_N\}$.  Then we have
\begin{eqnarray*}
(L_{l+1}-\psi(L_{l+1}))w'&=&\sum_{(\lambda, \mu)\notin\Lambda_N} p_{\lambda, \mu}
[L_{l+1},
L_{-\lambda}G_{\frac12-\mu}]w\\
&&+\sum_{\begin{subarray}
\ (\lambda, \mu)\in\Lambda_N\\
\ \ \lambda(l)=0
\end{subarray}} p_{\lambda, \mu}
[L_{l+1},
L_{-\lambda}]G_{\frac12}^iw\\
&&+\sum_{\begin{subarray}
\ (\lambda, \mu)\in\Lambda_N\\
\ \ \lambda(l)\neq0
\end{subarray}} p_{\lambda, \mu}
[L_{l+2}, L_{-\lambda}]G_{\frac12}^jw,
\end{eqnarray*} where $i,j=0, 1$.
Since the maxdeg's of the first summand and second summand are strictly smaller than $N-l$, and the maxdeg's of the third summand is $N-l$, we have $(L_{l+1}-\psi(L_{l+1}))w'\ne 0$.

Since the proof for $\ep=0$ case (iii) is similar, we omit the details.

\end{proof}

\begin{lem}\label{lem:predotActionFinite}
Let $\lambda \in \p$, $\mu\in\q$ and $E_n=L_n$ or $G_{n+1-\ep}$.
\begin{itemize}
\item[(i)] For all $n>1$, $E_n(L_{-\lambda}G_{\ep-\mu}w_{\psi}) \in
\mbox{span}_{\cc} \{ L_{-\lambda'}G_{\ep-\mu'} w_{\psi} \mid |
\mu'-\ep|+|\lambda'| + \lambda'(0) \leq |\lambda|+|\mu-\ep| + \lambda(0)\}$.
\item[(ii)]  If $n > |\lambda|+|\mu-\ep|+2$, then $E_n(L_{-\lambda}G_{\ep-\mu}w_{\psi}) = 0$.
\end{itemize}
\end{lem}
\begin{proof} By direct calculation as Lemma \ref{lemm-weight-1}.
\end{proof}

\begin{lem}\label{lemma-finite}
Let $V$ be a Whittaker module, and $v \in V$.
Regarding $V$ as a $\mathfrak{p}_{\ep}$-module, $U(\mathfrak{p}_{\ep})v$ is a finite-dimensional submodule of $V$.
\end{lem}
\begin{proof} It follows from Lemma
\ref{lem:predotActionFinite}.
\end{proof}

\begin{lem} \label{submodule}
Let $U$ be  a submodule of $W_{\ep}(\psi,c)$.  Then there is a nonzero Whittaker vector $u\in U$.
\end{lem}
\begin{proof} It follows from Lemma \ref{lemma-finite} and Proposition \ref{irr-finite-1}.\end{proof}

Our mains result can be stated as follows:

\begin{thm}\label{simple}  Suppose that $\psi:\mathfrak{p}_\ep\to\cc$ is a nontrivial Lie superalgebra homomorphism and $c\in\cc$. Then
\begin{itemize}
\item[(i)]$W_{\frac{1}{2}}(\psi, c)$ is simple.
\item[(ii)]$W_{0}(\psi, c)$ is simple if and only if $\psi(L_2)\neq0$.
\end{itemize}
\end{thm}
\begin{proof}

 For $\ep=\frac12$, if $\psi(L_2)=0$ and $\psi(L_1)\ne 0$, then $G_{\frac12}w_\psi$ is also a Whittaker vector. However the  Whittaker module generated by $G_{\frac12}w_\psi$ is same as $W_{\frac{1}{2}}(\psi,c)$ since $G_{\frac12}G_{\frac12}w_\psi=\psi(L_1)w_{\psi}$.
So (i) and the sufficient condition (ii) follow from Proposition \ref{whittaker-vector} and Lemma \ref{submodule}.

We only need to prove the necessary condition for (ii). If $\psi(L_2)=0$ and $\psi(L_1)\ne 0$, then $G_{1}w_\psi$ is also a Whittaker vector of $W_{0}(\psi, c)$. In this case the Whittaker module generated by $G_{1}w_\psi$ is a proper submodule of $W_0(\psi,c)$ since $w_\psi\not\in \langle G_1w_\psi\rangle$.
\end{proof}

\begin{rem}\label{rem45}
 For $\ep=0$, if $\psi(L_2)=0$ and $\psi(L_1)\ne 0$, the vector $G_1w_\psi$ is called the degenerate Whittaker vector \cite{DLM}. Moreover $W_0(\psi,c)/\langle G_1w \rangle$ is a Whittaker module of $(\phi,c)$, where  $\phi:\SVir_{0}^+ \to \cc$ such that $\phi(L_2)=0$ and $\phi(L_1)\ne 0$ (in this case $\phi(G_1)=0$ since $\phi(L_2)=\phi(G_1)^2$). Similar to the proof of Proposition \ref{whittaker-vector}, we can show that $W_0(\psi,c)/\langle G_1w \rangle$ is simple. Then in this case $L_0(\psi, c)=W_0(\psi,c)/\langle G_1w \rangle$.
\end{rem}

\begin{cor} \label{gen-sim}
Assume  $\psi:\mathfrak{p}_\ep\to\cc$ is a nontrivial Lie superalgebra homomorphism and $c\in\cc$. Let $M_\ep$ be
a simple Whittaker module of type $(\psi,c)$ for $\SVir_\ep$. Then $M_\ep\cong L_\ep(\psi,c)$ or $\Pi\, L_\ep(\psi,c)$.
\end{cor}
\begin{proof} Let $w\in M_\ep$ be a Whittaker vector
corresponding to $\psi$. If $|w|=|w_\psi|$ (the $\z_2$ degree's), then by the
universal property of $W_\ep({\psi}, c)$, there exists a module homomorphism
$\varphi : W_\ep({\psi}, c) \to M_\ep$ with $uw_\psi\mapsto uw$.  This map is
surjective since $w$ generates $M_\ep$. Then $M_\ep\cong L_\ep(\psi,c)$ since $M_\ep$ is simple. Similarly, if $|w|\ne|w_\psi|$, then $M_\ep\cong \Pi\, L_\ep(\psi,c)$. As a corollary of Theorem \ref{simple}, the proof is complete.
\end{proof}

\begin{cor}
Let $V$ be finite-dimensional module over $\SVir_\ep^+$ with $C$ acts as a scalar $c\in \cc$. Then  $M_\ep(V,c)$  is a simple module for $\SVir_\ep$  if and only if $V$ is simple module for $\SVir_\ep^+$.
\end{cor}

 \centerline{\bf ACKNOWLEDGMENTS}

\vskip15pt We gratefully acknowledge the partial financial support from the NNSF (No.11871249,\\ No.11771142), the ZJNSF(No.LZ14A010001), the Shanghai Natural Science Foundation (No.16ZR1425000) and the Jiangsu Natural Science Foundation(No.BK20171294). Part of this work was done while the authors were visiting the Chern
Institute of Mathematics, Tianjin, China. The authors would like to
thank the institute and Prof. Chengming Bai for their warm hospitality
and support. We also thank the referee for his/her helpful suggestions.

\end{document}